\documentclass[a4paper]{article}
\usepackage{fullpage}

\usepackage[usenames,dvipsnames]{xcolor}
\usepackage[pdftex]{hyperref}

\newcommand{\Rset}{\mathbb{R}}

\newcommand{\comment}[1]{}

\newtheorem{theo}{Theorem}
\newtheorem{defi}[theo]{Definition}

\newtheorem{coro}[theo]{Corollary}

\newtheorem{prop}[theo]{Proposition}
\DeclareSymbolFont{stmry}{U}{stmry}{m}{n}
\SetSymbolFont{stmry}{bold}{U}{stmry}{b}{n}
\DeclareMathDelimiter\llbracket{\mathopen}{stmry}{"4A}{stmry}{"71}
\DeclareMathDelimiter\rrbracket{\mathclose}{stmry}{"4B}{stmry}{"79}
\newenvironment{proof}{\noindent{\bf Proof.~}}
{{\mbox{}\hfill {\small \fbox{}}\\}}

\def\supp{{\mathrm {\rm supp}}}
\def\uu{\underline{u}}
\def\uuu{\underline{\underline{u}}}

\definecolor{vert}{rgb}{0.06, 0.7, 0.6}
\definecolor{mauve}{rgb}{0.7, 0.2, 0.99}

\usepackage{graphics} % for pdf, bitmapped graphics files
\usepackage{epsfig} % for postscript graphics files
\usepackage{mathptmx} % assumes new font selection scheme installed
\usepackage{times} % assumes new font selection scheme installed
\usepackage{amsmath} % assumes amsmath package installed
\usepackage{amssymb}  % assumes amsmath package installed

\title{Establishing Traveling Wave in Bistable Reaction-Diffusion System by Feedback}
\author{Pierre-Alexandre Bliman and Nicolas Vauchelet% <-this % stops a space
\thanks{This work was supported by Inria, France and CAPES, Brazil (processo 99999.007551/2015-00), in the framework of the STIC AmSud project MOSTICAW.
N.V.\ acknowledges support from the Emergence project from Mairie de Paris, {\it Analysis and simulation of optimal shapes - application to lifesciences}.}% <-this % stops a space
\thanks{Pierre-Alexandre Bliman is with Sorbonne Universit\'es, Inria, UPMC Univ. Paris 06, Lab.\ J.L.\ Lions UMR CNRS 7598, Paris, France and Escola de Matem\'atica Aplicada, Funda\c c\~ao Getulio Vargas, Rio de Janeiro - RJ, Brazil
        {\tt\small pierre-alexandre.bliman@inria.fr}}%
\thanks{Nicolas Vauchelet is with LAGA - UMR 7539,
Institut Galil\'ee,
Universit\'e Paris 13,
99, avenue Jean-Baptiste Cl\'ement,
93430 Villetaneuse, France
        {\tt\small vauchelet@math.univ-paris13.fr}}%
}

\begin{document}

\maketitle
%\thispagestyle{empty}
%\pagestyle{empty}

%%%%%%%%%%%%%%%%%%%%%%%%%%%%%%%%%%%%%%%%%%%%%%%%%%%%%%%%%%%%%%%%%%%%%%%%%%%%%%%%
\begin{abstract}

Several stains of the intracellular parasitic bacterium {\em Wolbachia} limit severely the competence of the mosquitoes {\em Aedes aegypti} as a vector of dengue fever and possibly other arboviroses.
For this reason, the release of mosquitoes infected by this bacterium in natural populations is presently considered a promising tool in the control of these diseases.
Following works by M.\ Turelli \cite{Barton:2011aa} and subsequently M.\ Strugarek et al.\ \cite{Strugarek:2016ab,Strugarek:2016aa}, we consider a simple scalar reaction-diffusion model describing the evolution of the proportion of infected mosquitoes, sufficient to reveal the bistable nature of the {\em Wolbachia} dynamics.
A simple distributed feedback law is proposed, whose application on a compact domain during finite time is shown to be sufficient to invade the whole space.
The corresponding stabilization result is established for any space dimension.
\end{abstract}

{\bf Keywords:} Biological systems; Distributed parameter systems; Distributed control

\section{Introduction}
\label{se1}

%\subsection{Control of dengue disease by introduction of the bacterium {\em Wolbachia} in vector populations}
\label{se11}

%Arboviruses are viruses transmitted to humans by arthropods, such as mosquitoes, that put at risk considerable portions of the human population and infect millions of people yearly.
%The control of diseases such as y
Dengue, chikungunya or zika fever put at risk considerable portions of the human population.
In absence of vaccine or curative treatment, acting on the population of mosquitoes {\em Aedes aegypti} that are their vectors is essentially the only feasible control method.
Application of insecticides and mechanical remotion of breeding sites are the most popular methods.
However, implementing the latter necessitates massive public campaigns with mixed efficiency, while beyond their negative impact on the environment and other species, intensive use of insecticides has induced gradual increase of the mosquito resistance and correlative efficiency decrease \cite{MacieldeFreitas2014,montella2007insecticide}.
Therefore alternative methods have been proposed and implemented.
Among them the release of transgenic or sterile mosquitoes has been tested \cite{Alphey2010,Alphey2014}.
The latter, based on local eradication of  the vector, suffers from intrinsic lack of robustness against subsequent 
reinvasions.

The release of {\em Aedes aegypti} mosquitoes deliberately infected in laboratory by the bacterium {\em Wolbachia} has been proposed recently as a promising strategy \cite{Walker2011,Hoffmann2011,Hancock2011,Alphey2013,Hughes2013}, due to the fact that it drastically limits the vectorial competence of the infected mosquitoes \cite{Moreira2009}.
{\em Wolbachia} is a maternally transmitted endo-symbiont, widely present in arthropods in nature, but not in {\em Aedes aegypti}. It is characterized by {\em cytoplasmic incompatibility}, the fact that a {\em Wolbachia}-free female fertilized by a {\em Wolbachia}-infected male does not produce viable offsprings \cite{Werren2008}.
%Moreover {\em Wolbachia}-infected mosquitoes cease to transmit some arboviruses like dengue 
%Thus, artifial injection of {\em Wolbachia}-infected mosquitoes in the field has been investigated.
Mathematical models have been proposed to study the biological invasion of the {\em Wolbachia}-infected population \cite{Barton:2011aa,Chan2013,Fenton2011,Hughes2013,Bliman:2015aa}.

Spacial invasion of a population is commonly modeled by reaction-diffusion system of equations.
Barton and Turelli \cite{Barton:2011aa} have shown the ability of the following reaction-diffusion system to describe {\em Wolbachia} invasion:
denoting $p(t,x)\in [0,1]$ the proportion of infected mosquitoes at time $t\geq 0$ in the point $x\in\Rset^d$, and $\sigma$ the diffusivity, the system reads
\begin{subequations}
\label{eq0}
\begin{gather}
\partial_t p - \sigma \Delta p = f(p),\qquad (t,x)\in [0,+\infty)\times \Rset^d,\\
p(0,.) = p^0 \in L^\infty(\Rset^d;[0,1]).
\end{gather}
\end{subequations}
Notice that system \eqref{eq0} may be recovered through reduction of a more complex model describing the evolution of each population of mosquitoes, under the assumption of large population \cite{Strugarek:2016ab}.
The function $f$, characteristic of the interactions between the two populations, is given by 
\begin{equation}
\label{eq:f}
f(p)= \delta d s_h \frac{p(1-p)(p-\theta)}{s_h p^2 - (s_f+s_h) p +1}, \qquad
\theta = \frac{s_f +\delta -1}{\delta s_h}.
\end{equation}
All constants are positive, and have the following meaning: $d$ is the death rate of the uninfected population, $\delta d$ is the death rate of the infected population ($\delta>1$); $s_f\in [0,1)$ characterizes the fecundity decrease ($(1-s_f)$ is the ratio between the fecundity in the infected and non-infected populations); $s_h\in (0,1]$ characterizes the completeness of the cytoplasmic incompatibility (a fraction $s_h$ of uninfected females eggs fertilized by infected males will not hatch --- in case of perfect CI, $s_h=1$).
We assume $s_f+\delta -1 < \delta s_h$, in such a way that $\theta\in (0,1)$.
In such conditions, the function $f$ is {\em bistable}, in the following precise sense.

\begin{defi}[Bistable function]
A continuous function $f\ :\ [0,+\infty)\to\Rset$ is called bistable if there exists $\theta\in (0,1)$ such that $f$ is null on $0, \theta$ and $1$, negative on $(0,\theta)$ and positive on $(\theta,1)$.
%\hfill\IEEEQEDopen
\end{defi}

Several types of traps exist to capture mosquitos, permitting to evaluate their abundance through statistical methods \cite{Focks:2003aa,Silver:2007aa}.
On the other hand, polymerase chain reaction (PCR) method is used to screen for the presence of the bacterium {\em Wolbachia} in the captured sample \cite{Hoffmann2011}.
%This potentially makes available the value of the state.
%On the other hand, the fact that these measures are of a discrete in time nature has been ignored in this first work on the subject
One may therefore consider that  measurements of the state $p(t,x)$ are available during the treatment process, and that it is possible to consider {\em feedback} control strategies for scheduling and dimensioning of the releases (in order to validate such principle, we disregard here the discrete in time nature of the measurement).
As usual, the expected advantage of feedback compared to open-loop approaches (where the release
schedule is computed once for all a priori), is its ability to cope with parametric and dynamic uncertainties on the model.
We propose and study in this paper a class of distributed feedback laws that guarantee the success of the invasion.
%More precisely, we consider that the release of mosquitoes is adapted according to the dynamics of the population.
A major feature is that the control law we propose acts on a fixed {\em bounded} domain, denoted $\Omega$ in the sequel, during a {\em limited} time $T>0$. 
More precisely, denoting $u$ the proportion of infected mosquitoes, the controlled system satisfies the following reaction-diffusion system, obtained from \eqref{eq0} by adding to the reaction term a distributed control term $g$ with support in $\Omega$ and taking nonnegative values only:
\begin{subequations}
\label{equ0}
\begin{gather}
\partial_t u - \sigma \Delta u = f(u) + g(u) \mathbf{1}_{[0,T]},  \\
u(0,.) = u^0 \in  L^\infty(\Rset^d; [0,1]).
\end{gather}
\end{subequations}
The main contribution of the present paper is to prove that there exists systematic way to choose a time $T>0$, a bounded domain $\Omega$, and a distributed control law $g(u)$ null outside the bounded domain $\Omega$, such that, {\em for any initial value $u^0$}, the solution to the control problem \eqref{equ0} satisfies $u(t,x) \to 1$ when $t\to +\infty$, {\em for any $x$ in $\Rset^d$}.
Moreover, we propose explicit expressions for these objects, see below the precise statement of the main result, Theorem \ref{th2}.

The outline of the paper is the following.
Some well-known results on reaction-diffusion systems, useful for the study, are recalled in the next section.
The main result is stated and illustrated by numerical examples in Section \ref{sec:main}, and afterwards proved in Section \ref{se3}.
Concluding remarks and open questions are exposed in Section \ref{sec:fin}.
Last, an appendix provides the proof of a sufficient condition for invasion in bistable systems.

\section{Some recall on reaction-diffusion systems}

For the sake of clarity and completeness of the paper, we first recall in this section some useful results on bistable reaction-diffusion systems (see e.g. \cite{Fife1979}).

\subsection{Comparison principle in parabolic systems}
\label{se21}

\begin{defi}[Subsolutions and supersolutions]
Let $\Omega \subset \Rset^d$ be a regular, open set (bounded or not). Let $T>0$. Let $f : \Rset \to \Rset$ and $h : \partial \Omega \to \Rset$ be two smooth functions.
We consider an elliptic operator $\mathcal{L} := \Delta + k(x) \nabla$, where $k$ is a smooth function $\Omega \to \Rset^d$.
A subsolution to the parabolic problem 
\begin{equation}
\begin{array}{l}
\partial_t u - \mathcal{L} u = f(u) \text{ in } \Omega, \
u(t, \cdot) = h(t, \cdot) \text{ on } (0,T)\times\partial \Omega, \\[2mm]
u(0, \cdot) = u^0 (\cdot) \text{ in } \Omega.
\end{array}
  \label{eq:parabolic}
\end{equation}
is a function $\underline{u}$ such that
\begin{equation}
\begin{array}{l}
\partial_t \underline{u} - \mathcal{L} \underline{u} \leq f(\underline{u}) \text{ in } \Omega, \
\underline{u}(t, \cdot) \leq h(t, \cdot) \text{ on } (0,T)\times\partial \Omega, \\[2mm]
\underline{u}(0, \cdot) \leq u^0 (\cdot) \text{ in } \Omega.
\end{array}
\end{equation}
Similarly, a super-solution of %the parabolic problem
\eqref{eq:parabolic} is a function $\overline{u}$ such that
\begin{equation}
\begin{array}{l}
\partial_t \overline{u} - \mathcal{L} \overline{u} \geq f(\overline{u}) \text{ in } \Omega, \
\overline{u}(t, \cdot) \geq h(t, \cdot) \text{ on } (0,T)\times \partial \Omega, \\[2mm]
\overline{u}(0, \cdot) \geq u^0 (\cdot) \text{ in } \Omega.
\end{array}
\end{equation}
By definition, a solution is any function which is simultaneously a sub- and a super-solution.
%\hfill\IEEEQEDopen
\end{defi}

Sub- and supersolutions are used in the classical comparison principle:
\begin{prop}[Parabolic comparison principle]
For all $T > 0$ we introduce the ``parabolic boundary''
$$
\partial_T \Omega :=  [0, T) \times \partial \Omega\ \bigcup\ \{ 0 \} \times \Omega.
$$

If $\underline{u}$ (resp. $\overline{u}$) is a sub-solution (resp. a super-solution) to \eqref{eq:parabolic}, and $u$ is a solution such that $u \geq \overline{u}$ (resp. $u \leq \underline{u}$) on $\partial_T \Omega$, then this inequality holds on $\Omega \times [0, T]$.
%\hfill\IEEEQEDopen
 \label{prop:comparison}
\end{prop}

\subsection{Traveling waves in bistable reaction-diffusion systems}
\label{se12}

Motivated by the previous example, we examine in this paper the question of onset of {\em traveling waves} in general system \eqref{eq0} with $f$ a bistable function.
Traveling waves are particular solutions of \eqref{eq0} of the type $p(t,x) := \tilde p(e\cdot x-ct)$ which connects the two stable steady states, i.e. $\tilde p(-\infty)=1,\ \tilde p(+\infty)=0$.
The normalized vector $e\in\Rset^d$ refers to the direction of propagation, and the quantity $c$ corresponds to the speed of the wave.
When $c>0$ the state $1$ (complete infestation by {\em Wolbachia} for the example developed in Section \ref{se11}) invades the states $0$, and vice versa.
Injecting the expression $p(t,x) = \tilde p(e\cdot x-ct)$ into \eqref{eq0}, we get
$0 = \partial_t p - \sigma \Delta p - f(p) = -c\tilde p' - \sigma\tilde p''-f(\tilde p)$.
Multiplying by $\tilde p'$ and integrating yields
\[
c\int_{-\infty}^{+\infty} (\tilde p'(z))^2\ dz = - \int_{-\infty}^{+\infty} f(\tilde p(z)) \tilde p'(z)\ dz = \int_0^1 f(z)\ dz,
\]
from which we deduce that the sign of $c$ is the same as the sign of $\int_0^1 f(z)\,dz$.
Then, in order to have evolution towards the equilibrium value $1$, it is {\em necessary} that
\begin{equation}
\label{eq2}
 \int_0^1 f(z)\ dz >0
\end{equation}
In consequence, we assume in all the paper that: %the following condition holds:
\begin{equation}
\label{eq1}
\exists \theta_c\in (\theta,1),\ F(\theta_c)=0\ \text{ with } F(z) := \int_0^z f(\xi)\ d\xi,\ z\in [0,1]
\end{equation}
This seems to be the case for the problem presented in Section \ref{se11}, see \cite{Barton:2011aa}.

The issue of onset of traveling waves in systems of type \eqref{eq0} with $f$ bistable fulfilling assumption \eqref{eq1} has been studied in \cite{Barton:2011aa,Strugarek:2016ab,Strugarek:2016aa}.
We now recall some key results.
%, first introducing the notion of {\em propagule}.

\begin{defi}[Propagule]
A {\em propagule} for equation \eqref{eq0} is any continuous initial function $p^0\ :\ \Rset^d\to [0,1]$ such that the corresponding solution $p$ of \eqref{eq0} verifies
\[
\forall x\in\Rset^d,\qquad \lim_{t\to +\infty} p(t,x) \text{ exists and equals } 1
\]
It is called {\em $\alpha$-propagule} if its supremum is equal to $\alpha$.
%\hfill\IEEEQEDopen
\end{defi}
Due to the comparison principle, {\em any initial condition bounded from below by a propagule is a propagule}.
Also, due to the homogeneity of the space in equation \eqref{eq0}, any translate of a propagule is a propagule.
To summarize, the set of propagules is an {\em upper set, invariant by translation}.

The following result answers the question of the existence of such objects.
It has been stated in \cite{Strugarek:2016aa}, as a consequence of \cite{Muratov:2017aa}. It relies on the existence of a threshold phenomena for the propagation in reaction-diffusion system as studied in \cite{Zlatos:2006aa,Du:2010aa,Polacik:2011aa,Muratov:2017aa}.
\begin{theo}[Existence of propagule \cite{Strugarek:2016aa}]
\label{th1}
Consider sys\-tem \eqref{eq0} with bistable function $f$ fulfilling \eqref{eq1}.
Then, for all $\alpha\in (\theta_c,1]$ there exists a compactly supported, non-increasing function $v_\alpha\ :\ \Rset^+\to\Rset^+$ with $v_\alpha(0)=\alpha$ such that, for any solution $p$ of \eqref{eq0} whose initial condition $p^0$ verifies:
\begin{equation}
\exists x_0\in\Rset^d,\ \forall x\in\Rset^d,\qquad
p^0(x) \geq v_\alpha(|x-x_0|),
\end{equation}
one has:
\begin{equation}
\lim_{t\to +\infty} p(t,x) = 1,
\end{equation}
for any $x\in\Rset^d$, locally uniformly.
Moreover, one can take the support of $v_\alpha$ in $[0,R_\alpha]$ with
\begin{equation}
\label{eq3}
R_\alpha := \left( \left(
1+ \frac{2F(\alpha)}{\sigma\alpha^2-2F(\theta)}
\right)^{1/d}-1
\right)^{-1} +1.
\end{equation}
%\hfill\IEEEQEDopen
\end{theo}
Here and in the sequel, the ``locally uniformly convergence" means convergence in $L^\infty$ on any compact set of $\Rset^d$.
%For sake of completeness, a proof of Theorem \ref{th1} is given in Appendix.
Notice that $R_\alpha\to +\infty$ when $\alpha\to \theta_c$.
The estimate of $R_\alpha$ in \eqref{eq3} is not optimal.
As a matter of fact the issue of optimality of the support of propagules is still an open question.

%\subsection{Onset of traveling waves by feedback control}
%\label{se13}

\section{Igniting traveling waves by feedback control}\label{sec:main}
%\section{PRINCIPLE OF THE FEEDBACK-LAW AND MAIN RESULTS}

\subsection{Main result}

In order to ignite the propagation of the traveling wave, we propose to impose during a finite time $T>0$ a feedback-law in an open bounded region $\Omega$ of the space $\Rset^d$.
For simplicity, the feedback-law will be chosen in such a way that the resulting closed-loop system is linear on $\Omega$.
More precisely, we consider that the function $g$ in \eqref{equ0} reads
\begin{equation}\label{eq:g}
g(u) = (\mu (1-u) - f(u))_+\ \mathbf{1}_\Omega, \qquad \mu > 0.
\end{equation}
The notation $(\cdot)_+$ is for the positive part.
Notice that the positive part is taken to guarantee the nonnegativity of the control function $g$.
We have that $g(1)=0$, meaning that there is no action wherever the desired proportion $u=1$ is attained. 

Therefore, the controlled system under study is as follows
\begin{subequations}
\label{eq5}
\begin{equation}
\partial_t u - \sigma\Delta u = \mu (1-u)\ \text{ on } [0,T]\times (\Omega\cap\{g>0\} ),
\end{equation}
\vspace{-0.4cm}
\begin{equation}
\partial_t u - \sigma\Delta u = f(u)\ \text{ on } [0,T]\times \left(
(\Rset^d\setminus\overline\Omega)\ \bigcup\ \{g\leq 0\}
\right)
\bigcup\ (T,+\infty)\times\Rset^d,
\end{equation}
\vspace{-0.4cm}
\begin{equation}
u(0,\cdot )= u^0.
\end{equation}
\end{subequations}
The initial condition $u^0$ takes on values in $[0,1]$ and is typically zero in the problem of infestation by {\em Wolbachia} previously described, corresponding to the situation where initially no mosquito is infected.

The main result of this paper is the following.
\begin{theo}[Main convergence result]
\label{th2}
Let $f$ be a bistable function.
Then, for any $\mu>0$, there exist $T>0$ and a bounded open set $\Omega\subset \Rset^d$ such that all solutions to \eqref{eq5} converge to $1$ as $t$ goes to $+ \infty$, locally uniformly on $\Rset^d$.
Also, for any $\alpha,\overline{\alpha}$ such that $\theta_c < \alpha < \overline{\alpha} <1$ it is sufficient, in order to have convergence, to choose $T$ such that
\begin{equation}
\label{cond2}
T \geq \frac{1}{\mu} \ln\left(
\frac{\overline{\alpha}}{\overline{\alpha}-\alpha}
\right)
\end{equation}
and $\Omega$ containing a ball of radius $(1+\varepsilon^*(\alpha,\overline{\alpha}))R_\alpha$ where $R_\alpha$ is given in \eqref{eq3} and
\begin{equation}
\label{eq11}
\varepsilon^*(\alpha,\overline{\alpha}) :=
\frac{8}{\sqrt{(d-1)^2+\frac{32}{3}\frac{R_\alpha^2\mu(1-\overline{\alpha})}{\sigma\overline{\alpha}}} -d+1}.
\end{equation}
Last, the solutions of \eqref{eq5} are increasing with respect to $\mu,T$ and $\Omega$ (relatively to the order defined by the inclusion).
%\hfill\IEEEQEDopen
\end{theo}

This result states that, given a feedback function $g$ as above (for fixed $\mu>0$), there exist a time of control $T$ and a domain $\Omega$ such that the proposed feedback control yields invasion of the {\em Wolbachia}-infected population.
The proof of this result relies on the construction of a subsolution to \eqref{eq5}, itself located above a propagule (whose existence is established by Theorem \ref{th1}). Then the comparison principle will yield the result. The construction of such a subsolution can be made explicit, leading to conditions \eqref{cond2} and \eqref{eq11}.
Notice that both these formulas involve a free parameter $\overline{\alpha}$.
The latter may be optimized to fit some requirement. For instance, if $T$ is required to be as small as possible, we may choose $\overline{\alpha}$ as big as possible, i.e.\ close to $1$. However, when $\overline{\alpha}$ is close to $1$, one sees from \eqref{eq11} that the domain $\Omega$ should contain a ball with a radius going to $+\infty$.

Instead of fixing the feedback gain $\mu$, as done in Theorem \ref{th2}, it is also possible to prescribe a maximal control time or a bounded release domain.
Corresponding reformulations are stated in the two following corollaries.

\begin{coro}[Prescribed maximal control time]
\label{co1}
Let $f$ be a bistable function, then for any $T>0$, there exist a domain $\Omega$ and $\mu>0$ such that all solutions to \eqref{eq5} converge to $1$ as $t$ goes to $+ \infty$, locally uniformly on $\Rset^d$.
%\hfill\IEEEQEDopen
\end{coro}

\begin{proof}%[Proof of Corollary \ref{co1}]
Fix $\alpha$ and $\overline{\alpha}$ such that $\theta_c<\alpha<\overline{\alpha}<1$.
Take $\mu>0$ such that \eqref{cond2} is satisfied, i.e.\ $\mu \geq \frac{1}{T} \ln (\overline{\alpha}/(\alpha-\overline{\alpha}))$.
Then Theorem \ref{th2} applies and provides the desired result provided that $\Omega$ contains a ball of radius $(1+\varepsilon^*(\alpha,\overline{\alpha}))R_\alpha$, with $R_\alpha$ and $\varepsilon^*(\alpha,\overline{\alpha})$ given respectively in \eqref{eq3} and \eqref{eq11}.
\end{proof}

\begin{coro}[Prescribed bounded release domain]
\label{co2}
Let $f$ be a bistable function and a bounded domain $\Omega$ such that $B_{(1+\varepsilon)R_\alpha} \subset \Omega$ for some $\alpha\in(\theta_c,1)$ and $\varepsilon >0$, where $R_\alpha$ is the radius of the support of a propagule in Theorem \ref{th1} (for instance $R_\alpha$ may be as in \eqref{eq3}). 
Then, there exist $T>0$ and $\mu>0$ such that all solutions to \eqref{eq5} converge to $1$ as $t$ goes to $+ \infty$, locally uniformly on $\Rset^d$.
%\hfill\IEEEQEDopen
\end{coro}

\begin{proof}%[Proof of Corollary \ref{co2}]
Choose $\mu$ such that for any $\overline{\alpha}\in (\alpha,1)$, we have 
$$
\mu \geq \frac{\sigma \overline{\alpha}}{(1-\overline{\alpha})R_\alpha^2} \left(\frac{2}{\varepsilon^2} + \frac{1}{2\varepsilon} (d-1)\right).
$$
With this value, Theorem \ref{th2} applies with $T$ as in \eqref{cond2}.
\end{proof}

Last, notice that the speed of convergence towards the fully infested state depends upon the speed $c$ of the wave and the traveling wave profile, see e.g.\ \cite[Theorem 1]{Muratov:2017aa}.

\subsection{Numerical illustration}

In order to illustrate the main result, we present a numerical example in one spacial dimension.
The numerical values, taken from \cite{Barton:2011aa}, are chosen as $s_f=0.1$, $s_h=0.3$, $\delta=1$, then $\theta = \frac{s_f}{s_h}$ in the expression of $f$ in \eqref{eq:f}.
System \eqref{eq5} is solved by discretization with an implicit finite difference scheme on the computational domain $[-20,20]$.

The numerical results are presented in the Figures below. 
We display the time dynamics of the proportion of infected population in Fig.\ \ref{fig}, \ref{fig2}, \ref{fig3}, \ref{fig5}.
In all simulations, the control time is fixed to $T=10$.
For the sake of comparison, feedbacks with different parameters have been tested. 
In Fig.\ \ref{fig}, the control gain is $\mu=0.5$ and the domain control $\Omega=[-1,1]$.
We observe that invasion occurs, showing that this control allows to pass from the steady state $0$ to the steady state $1$.
In Fig.\ \ref{fig1}  is displayed a zoom on $[-2,2]\times[0,10]$ of the time dynamics of the control input.
As expected this function is compactly supported in $\Omega$, and close to $0$ near the final control time $T$.
In Fig.\ \ref{fig2}, the control domain is changed to $\Omega=[-0.5,0.5]$. In Fig.\ \ref{fig3},
$\mu$ is changed to $\mu=0.15$.
In both case, the control is not sufficient to guarantee invasion.
Finally, in Fig.\ \ref{fig5}, $\mu=0.15$ and $\Omega=[-2,2]$ and the spread of the infected population is observed.
%The essays clearly attest to the existence of a threshold between extinction and invasion \cite{Zlatos:2006aa,Du:2010aa,Polacik:2011aa,Muratov:2013aa,Muratov:2017aa}.
The essays clearly attest to the fact that a threshold exists, below which the infection by {\em Wolbachia} gets extinguished, and above which it invades the population \cite{Zlatos:2006aa,Du:2010aa,Polacik:2011aa,Muratov:2017aa}.
Also, it is apparent that the wave progresses with linear speed, independent from the control parameters (compare Fig.\ \ref{fig} and \ref{fig5}).
Last, it can also be noticed that the solution is a nondecreasing function of $\mu, T$ and $\Omega$.

\begin{figure}
\centering\includegraphics[width=\linewidth]{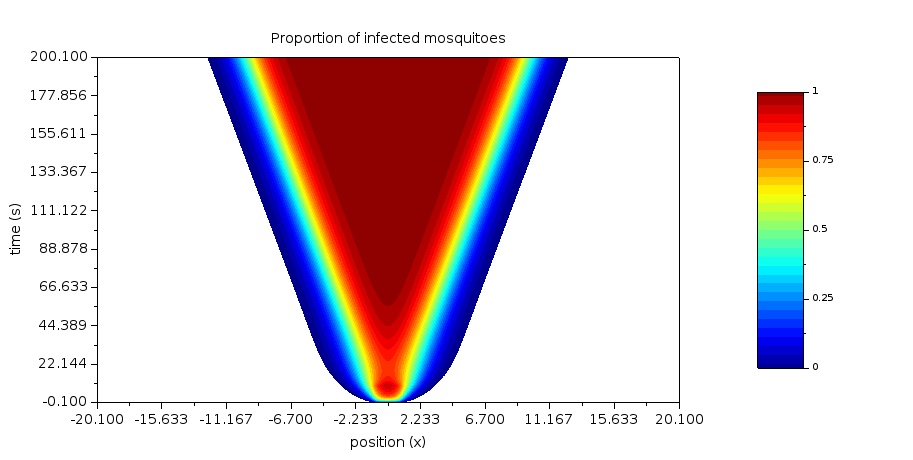}
%\centering\includegraphics[scale=0.25]{control2D1.eps}
\caption{Evolution of the proportion of infected (in $x$-axis) as a function of time ($y$-axis) for system \eqref{eq5}.
Parameters are $T=10$, $\mu=0.5$, $\Omega=[-1,1]$.}
\label{fig}
\end{figure}

\begin{figure}
%\centering\includegraphics[scale=0.25]{controlfunction2D1_zoom.eps}
\centering\includegraphics[width=\linewidth]{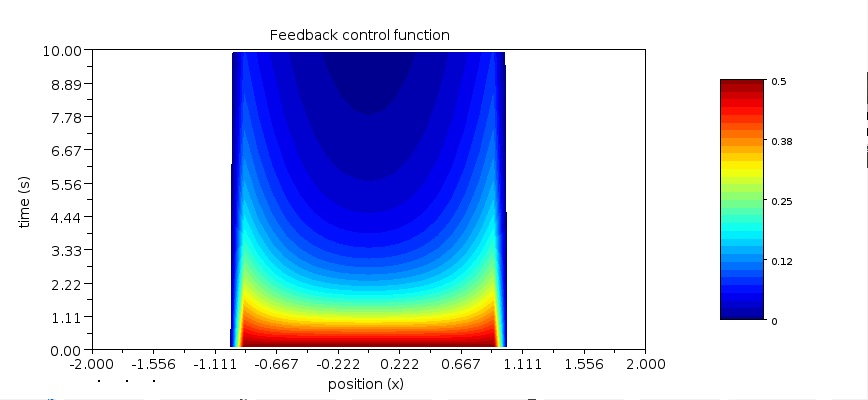}
\caption{Evolution of the control input ($x$-axis) as a function of time ($y$-axis). Same parameter set than in Fig.\ \ref{fig}.}
\label{fig1}
\end{figure}

\begin{figure}
%\centering\includegraphics[scale=0.25]{control2D2_L05.eps}
\centering\includegraphics[width=\linewidth]{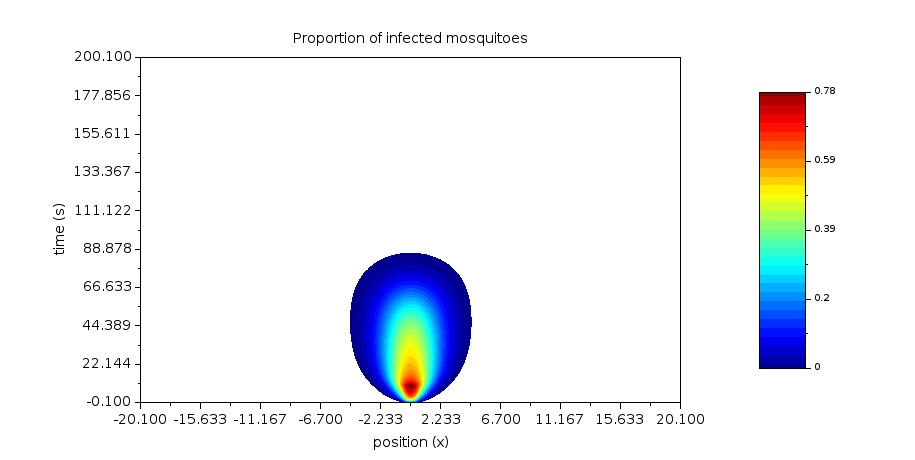}
\caption{Similar to Fig.\ \ref{fig}, with $T=10$, $\mu=0.5$, $\Omega=[-0.5,0.5]$.}
%Time dynamics of the proportion of infected population with a feedback control function as in system \eqref{eq5}. {\color{blue} The control time is fixed to $T=10$, $\mu=0.5$, $\Omega=[-0.5,0.5]$}}
\label{fig2}
\end{figure}

\begin{figure}
%\centering\includegraphics[scale=0.25]{control2D2_mu015.eps}
\centering\includegraphics[width=\linewidth]{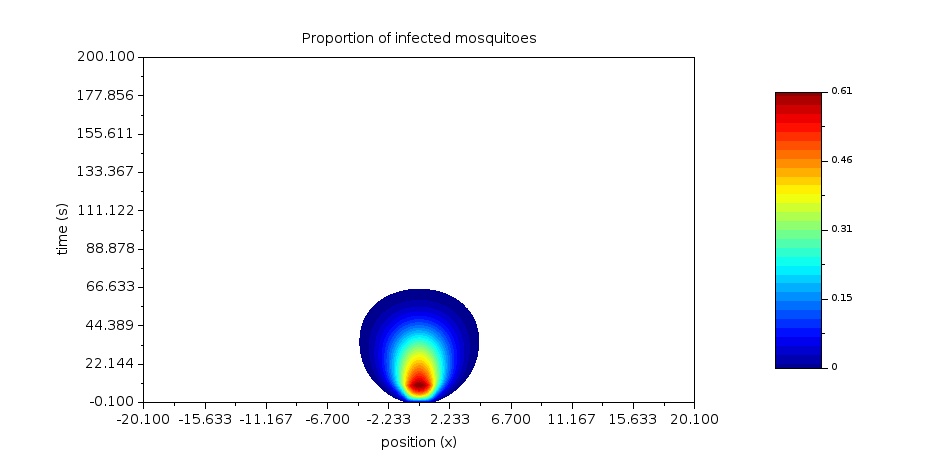}
\caption{Similar to Fig.\ \ref{fig}, with $T=10$, $\mu=0.15$, $\Omega=[-1,1]$.}
\label{fig3}
\end{figure}

% \begin{figure}
% \centering\includegraphics[width=\linewidth]{control2D2_mu1.jpg}
% \caption{Time dynamics of the proportion of infected population with a feedback control function as in system \eqref{eq5}. {\color{blue} The control time is fixed to $T=10$, we take $\mu=1$ and $\Omega=[-1,1]$}}
% \label{fig4}
% \end{figure}

\begin{figure}
%\centering\includegraphics[scale=0.25]{control2D2_mu015L2.eps}
\centering\includegraphics[width=\linewidth]{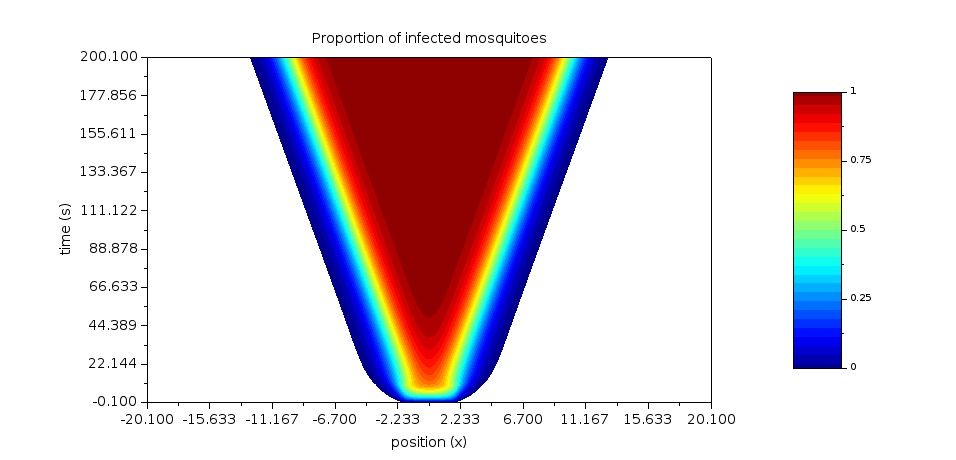}
\caption{Similar to Fig.\ \ref{fig}, with $T=10$, $\mu=0.15$, $\Omega=[-2,2]$.}
%\caption{Time dynamics of the proportion of infected population with a feedback control function as in system \eqref{eq5}. {\color{blue} The control time is fixed to $T=10$, we take $\mu=0.15$ and $\Omega=[-2,2]$}}
\label{fig5}
\end{figure}

% \begin{figure}
% \includegraphics[width=0.49\linewidth]{control1000.jpg}
% \includegraphics[width=0.49\linewidth]{control1002.jpg}\\
% \includegraphics[width=0.49\linewidth]{control1010.jpg}
% \includegraphics[width=0.49\linewidth]{control1020.jpg}\\
% \includegraphics[width=0.49\linewidth]{control1030.jpg}
% \includegraphics[width=0.49\linewidth]{control1040.jpg}
% \caption{Time dynamics of the proportion of infected population with a feedback control function as in system \eqref{eq5}. From top left to bottom right, the plotted times are : initial time, t=10, t=50, t=100, t=150, t=200.}
% \label{fig}
% \end{figure}

\section{Proof of Theorem \ref{th2}}
\label{se3}
%\mbox{}
%\subsection{Proof of Theorem \ref{th2}}
%\subsection{Proof of the convergence}
%\label{se31}

\noindent $\bullet$
We first prove the convergence result.
The proof is based on the construction of several auxiliary functions, permitting comparison with a solution $u$ of \eqref{eq5}.
We will more precisely proceed as follows.
Choose $\alpha\in (\theta_c,1)$.
Due to Theorem \ref{th1}, there exists an $\alpha$-propagule $v_\alpha$ with support contained in $[0,R_\alpha]$ centered in zero with radius $R_\alpha$.
We will show successively in the sequel that it is possible to find $T>0$, $\varepsilon>0$ and two nonnegative functions $\uu$ and $\uuu$ with support in $B_{(1+\varepsilon)R_\alpha}$ (all these objects depending upon $\alpha$) such that
\begin{enumerate}
\item
\label{pt1}
for any $t\in [0,T]$ and any $x\in\Rset^d$, $\uu(t,x) \leq u(t,x)$;
\item
\label{pt2}
for any $t \in [0,T]$ and any $x\in\Rset^d$, $\uuu(t,x) \leq \uu(t,x)$;
\item
\label{pt3}
for any $x\in B_{R_\alpha}$, $\displaystyle \uuu(T,x) \geq \alpha = \max_{x\in B_{R_\alpha}} v_\alpha(|x|)$.
%for any $t\geq T$ and any $x\in B_{R_\alpha}$, $\displaystyle \max_{x\in B_{R_\alpha}} v_\alpha(|x|) = \alpha \leq \uuu(t,x)$.
\end{enumerate}
These three properties imply that, for any $x\in B_{R_\alpha}$,
\begin{equation}
\label{eq7}
v_\alpha(|x|) \leq \uuu(T,x) \leq \uu(T,x) \leq u(T,x)
\end{equation}
As $\supp\ v_\alpha \subset [0,R_\alpha]$, one concludes that indeed $v_\alpha(|x|) \leq u(T,x)$ in the whole space $\Rset^d$.
Using the fact that $v_\alpha$ is a propagule, this demonstrates the convergence result in Theorem \ref{th2}, by applying Theorem \ref{th1}.
Therefore, it now only remains to prove the three points above.

\paragraph{Proof of point \ref{pt1}.}

For any $T>0$ and any $\varepsilon>0$, one may introduce the function $\underline{u}$, solution to the problem
\begin{subequations}
\label{eq6}
\begin{gather}
\partial_t \underline{u} - \sigma\Delta \underline{u} = \mu (1-\underline{u}), \qquad \text{ on } [0,T] \times B_{(1+\varepsilon)R_\alpha}, \\
\underline{u} = 0, \qquad \text{ on } [0,T] \times \partial B_{(1+\varepsilon)R_\alpha},  \\
\underline{u}(0,\cdot) \equiv 0, \qquad\text{ on } B_{(1+\varepsilon)R_\alpha}.
\end{gather}
\end{subequations}
The function $\uu$ thus defined is a subsolution for \eqref{eq5} on $[0,T]$ provided that $B_{(1+\varepsilon)R_\alpha} \subset \Omega$. Indeed, we have $u\geq 0$ on $\partial B_{(1+\varepsilon)R_\alpha}$, and on the set $\{g\leq 0\}$ we have $f(u)\geq \mu(1-u)$.
Due to the comparison principle, we deduce that $\uu\leq u$ on $B_{(1+\varepsilon)R_\alpha}$. Then, we extend $\uu$ by the constant $0$ on $\Rset^d\setminus B_{(1+\varepsilon)R_\alpha}$ and the point \ref{pt1}.\ is proved.

\paragraph{Proof of point \ref{pt2}.}
We will now construct the function $\uuu$.
Let us first introduce a function $\phi\in C^2([0,1])$ such that $\phi(0)=1$, $\phi(1)=0$, $\phi'(0)=\phi'(1)=0$ and $\phi'\leq 0$ on $[0,1]$.
Such function exists, take for instance the polynomial
\begin{equation}
\label{eq8}
\phi(x)=-2(1-x)^3+3(1-x)^2.
\end{equation}

Let now $\overline{\alpha} \in (\alpha,1)$ and introduce the radially symmetric nonincreasing function 
$$
\gamma(r) := \overline{\alpha}\ \mathbf{1}_{r\leq R_\alpha} + \overline{\alpha} \phi\left(
\frac{r-R_\alpha}{\varepsilon R_\alpha}
\right) \mathbf{1}_{R_\alpha<r\leq (1+\varepsilon) R_\alpha}.
$$
Clearly, $\gamma$ is non increasing on $[0,+\infty)$.
Moreover, its value is $\overline{\alpha}$ on $[0, R_\alpha]$ and its support is equal to $[0,(1+\varepsilon)R_\alpha]$.
In particular, for any nonnegative $r$, $0 \leq \gamma(r) \leq \overline{\alpha}$.

By definition $\gamma\in C^1(0,+\infty)$ and, except possibly for $r=R_\alpha$ and $(1+\varepsilon)R_\alpha$, we may compute its Laplacian.
For any $r\in (R_\alpha,(1+\varepsilon)R_\alpha)$, the latter is equal to  
\begin{eqnarray*}
-\Delta \gamma
& = &
-\partial_{rr} \gamma - \frac{d-1}{r} \partial_r \gamma(r)\\
& = &
-\frac{\overline{\alpha}}{\varepsilon^2 R_\alpha^2} \phi''\left(
\frac{r-R_\alpha}{\varepsilon R_\alpha}
\right) 
- \frac{\overline{\alpha}(d-1)}{\varepsilon R_\alpha r}\phi'\left(
\frac{r-R_\alpha}{\varepsilon R_\alpha}
\right),
\end{eqnarray*}
and it is equal to zero on $(0,R_\alpha)\cup ((1+\varepsilon)R_\alpha,+\infty)$.

Pick now $\varepsilon>0$ such that
\begin{equation}
\label{cond1}
\frac{1}{\varepsilon^2} \sup_{(0,1)}{|\phi''|}
+ \frac{d-1}{\varepsilon} \sup_{(0,1)}{|\phi'|} \leq \frac{R_\alpha^2\mu (1-\overline{\alpha})}{\sigma\overline{\alpha}}.
\end{equation}
This is possible, since $\phi\in C^2(0,1)$ and $\overline{\alpha}<1$.
With such a choice of $\varepsilon$, one has, for all $r\in(R_\alpha,(1+\varepsilon)R_\alpha)$,
\begin{equation}
\label{deltagamma}
|\sigma \Delta \gamma(r)|
\leq
\frac{\sigma\overline{\alpha}}{\varepsilon^2R_\alpha^2} \sup_{(0,1)}{|\phi''|}
+ \frac{\sigma\overline{\alpha}(d-1)}{\varepsilon R_\alpha r} \sup_{(0,1)}{|\phi'|}
\leq \mu (1-\overline{\alpha})
\leq \mu (1-\gamma(r)).
\end{equation}
%As a matter of fact, this condition is satisfied for all $r\in(R_\alpha,(1+\varepsilon)R_\alpha)$ provided that
The last inequality is deduced from the fact that $0\leq \gamma(r)\leq \overline{\alpha}$ everywhere.
Notice that since $\gamma$ is constant on $B_{R_\alpha}$, we have $\Delta \gamma=0$ and inequality \eqref{deltagamma} also holds true on $B_{R_\alpha}$.

We define now $\uuu$, as
\begin{equation}\label{def:uuu}
\uuu (t,x) := (1-e^{-\mu t}) \gamma(|x|).
\end{equation}
We compute
\begin{eqnarray}
\partial_t \uuu  - \sigma \Delta \uuu
& = &
\nonumber
\mu e^{-\mu t} \gamma - (1-e^{-\mu t}) \sigma \Delta \gamma\\
& \leq &
\nonumber
\mu e^{-\mu t} \gamma + \mu (1-e^{-\mu t}) (1-\gamma) \\
& = &
\label{eq9}
\mu(1-\uuu - e^{-\mu t} (1-\gamma)
\leq \mu (1-\uuu ).
\end{eqnarray}
Formula \eqref{deltagamma} was used to deduce the first inequality, and the fact that $\gamma\leq \overline\alpha \leq 1$ to deduce the second one.
Moreover, by definition of $\gamma$, we have $\uuu (0,.)\equiv 0$, and $\uuu (\cdot,x)=0$ for any $x\in\partial B_{(1+\varepsilon)R_\alpha}$.
Then $\uuu $ is a subsolution for \eqref{eq6}.
Applying the comparison principle, we deduce point \ref{pt2}.

\paragraph{Proof of point \ref{pt3}.}

Now, notice that from point \ref{pt2} and from the definition of $\uuu$ in \eqref{def:uuu}, we have
\begin{equation}
\label{eq10}
\forall x \in B_{R_\alpha},\ \forall t\in [0,T],\quad
\uuu(t,x) = (1-e^{-\mu t}) \overline{\alpha} \leq \underline{u}(t,x).
\end{equation}
Choose $T$ such that \eqref{cond2} is fulfilled.
For such a choice, one deduces from \eqref{eq10} that $\uuu(T,x)\geq\alpha$ for any $x\in B_{R_\alpha}$.
This proves the point \ref{pt3}.\ and concludes the proof of the (locally uniform) convergence towards 1 contained in Theorem \ref{th2}.\\

%\subsection{Proof of the estimates}
%\mbox{}

\noindent $\bullet$
We now demonstrate the estimates contained in the statement of Theorem \ref{th2}.
The estimate on $T$ comes from \eqref{cond2}, see above.
On the other hand, if $\phi$ in the beginning of the present proof is taken as in \eqref{eq8}, then $\sup_{[0,1]} |\phi'|=\frac 32$ and $\sup_{[0,1]} |\phi''|=6$, and condition \eqref{cond1} reads
$$
\frac{6}{\varepsilon^2} + \frac{2(d-1)}{3\varepsilon} \leq \frac{R_\alpha^2\mu (1-\overline{\alpha})}{\sigma\overline{\alpha}}
$$
that is $\varepsilon \geq \varepsilon^*(\alpha,\overline{\alpha})$ defined in \eqref{eq11}.

\noindent $\bullet$
The monotonicity of the solution with respect to
%with respect to the feedback parameters
$\mu, T, \Omega$ is a direct consequence of the comparison principle in Section \ref{se21}.
This finally achieves the proof of Theorem \ref{th2}.
%\hfill\IEEEQEDopen

\section{Conclusion and open questions}\label{sec:fin}

In this paper, we have studied the use of feedback control in a release protocol, in order to guarantee invasion of a  host population in bistable reaction-diffusion models. Our application example concerns the invasion of the maternally transmitted bacteria {\em Wolbachia} in populations of mosquitoes. The use of the latter is motivated by its blocking action on the transmission of some arboviruses like dengue.
We exhibit a class of feedback control functions which, when applied on a bounded domain $\Omega$ during finite time $T$, allows to pass from a {\em Wolbachia}-free population to fully {\em Wolbachia}-infected population as time goes to $+\infty$.

Several perspectives may be investigated in the future. First, as mentionned above, optimizing the release protocol is an important issue. Indeed, the conditions given in Theorem \ref{th2} are only sufficient, and may be improved depending on the constraints to be satisfied. For instance, one may be interested in minimizing the global number of mosquitoes introduced, or the treatment duration, or again the size of the release domain. However, the propagules functions introduced in Theorem \ref{th1} are not optimal. The construction of optimal functions igniting the propagation is still an open question.

Secondly, the mathematical model used in this study, dealing with the proportion of {\em Wolbachia}-infected mosquitoes, is a simplified version of a more elaborated model for two species ({\em Wolbachia}-infected and {\em Wolbachia}-free mosquitoes) \cite{Strugarek:2016ab}. Models including more biological features may also be encountered, for instance considering the different stages in the life of mosquitoes (larvae, eggs, pupae, adults), see \cite{Bliman:2015aa} and references therein. An interesting extension of the present work may be the study of a control on such more elaborated models.

Finally, we underline the fact that the environment is assumed homogeneous in the present work. Heterogeneity in the environment may have crucial consequences in the spread of population. In fact, stable fronts or blockings have been observed  \cite{Yeap2011,Hoffmann2014,Strugarek2017}. Our study does not take into account these phenomena and the use of a feedback control function to allow the crossing of potential barriers is a  direction of research that will be investigated.

\appendix

\section*{Appendix -- Proof of Theorem \ref{th1} \cite{Strugarek:2016aa}}

The approach is based on the energy method proposed in \cite{Muratov:2017aa}.
For sufficiently smooth function $p(t,x)$, define the energy as:
%\begin{equation}
%\label{Energy}
$E[p](t) = \int_{\Rset^d} \big( \frac{\sigma}{2} \lvert \nabla p(t,x) \rvert^2 - F(p(t,x))\big) dx$.
%\end{equation}
For any solution $p$ of \eqref{eq0},
%\begin{equation}
$\frac{d}{dt} E[p](t) = - \int_{\Rset^d} \big( \sigma \Delta p(t,x) + f(p(t,x))\big)^2 \,dx\leq 0$,
%\end{equation}
thus $E[p](t)\leq E[p^0]$ for any $t\geq 0$, where $p^0$ are the initial data for $p$.
Moreover, \cite[Theorem 2]{Muratov:2017aa} states that $p(t,\cdot)\to 1$ locally uniformly in $\Rset^d$ as $t\to +\infty$, provided that $\displaystyle\lim_{t\to +\infty} E[p(t,\cdot)] <0$.
Since $t\mapsto E[p(t,\cdot)]$ is nonincreasing, it suffices to construct $p^0$ s.t.\ $E[p^0]<0$ to prove Theorem \ref{th1}.

Let $\alpha > \theta_c$, and consider the family of initial data 
radially symmetric non-increasing along the rays and compactly supported in $B_{R+1}$, $R>0$,
defined by
$\phi_R(|x|)= \alpha$ if $0\leq |x| \leq R$ and $\phi_R(|x|)=\alpha(R+1-|x|)$ if $R \leq |x| \leq R+1$.
%We may compute
Then
$$
E[\phi_R]=|S_{d-1}| \int_0^\infty \left(
\frac{\sigma}{2} |\phi'_R(r)|^2 - F(\phi_R(r))
\right) r^{d-1} dr,
$$
$|S_{d-1}|$ being the volume of the unit sphere in $\Rset^d$.
By definition of $\phi_R$,
$E[\phi_R]/|S_{d-1}|
=
- \int_0^R F(\alpha) r^{d-1} dr
+ \int_R^{R+1} \left(
\frac{\sigma \alpha^2}{2} - F(\phi_R)
\right) r^{d-1}dr$.
As $f$ is bistable, $F(\theta)$, the minimal value of $F$ on $[0,1]$, is negative.
Therefore,
$d E[\phi_R] / R^d |S_{d-1}|
<
- F(\alpha)
+ \left(
\frac{\sigma\alpha^2}{2}-F(\theta)
\right)
\left(
\left(
1+\frac{1}{R}
\right)^d-1
\right)$.
Now $F(\alpha)>0$, since $\alpha>\theta_c$. 
We deduce that if $R+1\geq R_\alpha$ (such that supp$(\phi_R) \subset B_{R_\alpha}$),
where $R_\alpha$ is defined in \eqref{eq3}, then $E[\phi_R]<0$.
This achieves the proof of Theorem \ref{th1}.
%\hfill\IEEEQEDopen

%\IEEEtriggeratref{11}
\bibliographystyle{plain}
\bibliography{active}

\end{document}